\documentclass[preprint,12pt,number]{elsarticle}

\usepackage{amssymb}
\usepackage{amsmath}
\usepackage{amsthm}
\usepackage{mathtools}
\usepackage{url}
\usepackage[margin=1.1in]{geometry}
\usepackage[colorlinks=true,linkcolor=blue,citecolor=blue,urlcolor=blue]{hyperref}
\usepackage{booktabs}

\newtheorem{theorem}{Theorem}
\newtheorem{lemma}[theorem]{Lemma}

\newtheorem{corollary}[theorem]{Corollary}
\newtheorem{conjecture}[theorem]{Conjecture}
\theoremstyle{remark}

\newcommand{\Lr}{\operatorname{L}}

\begin{document}
\begin{frontmatter}

\title{On a conjecture due to Kanade related to Nahm sums}

\author[Berk]{Cetin Hakimoglu-Brown}
\cortext[mycorrespondingauthor]{Corresponding author}
\ead{mathemails@proton.me}
\address[Berk]{Berkeley, CA, USA}

\begin{abstract}
Kanade explored the construction of modular companions to $q$-series identities,
using the asymptotics of Nahm sums, and Mizuno [Ramanujan J. 66 (2025), Article 62, 31 pp.] recently obtained a generalization of Kanade's asymptotic formula
for symmetrizable Nahm sums. A related conjecture from Kanade concerning the
dilogarithm function and related to the work of Kur\c{s}ung\"oz on
Andrews--Gordon-type series [Ann. Comb. 23 (2019), 835--888] has remained open.
In this paper, we prove Kanade's conjecture, through an application of
dilogarithm identities due to Kirillov together with a dilogarithm ladder due
to Loxton and Lewin. Inspired by  Kanade's result, we extend this to prove two new related dilogarithm identities.

\end{abstract}


\end{frontmatter}

\section{Introduction}

The Rogers dilogarithm is:
\begin{equation}\label{eq:Ldef}
\Lr(x)=\operatorname{Li}_2(x)+\frac12\log x\log(1-x),
\qquad 0<x<1,
\end{equation}
with $\Lr(1)=\pi^2/6$.  We shall use Euler's reflection relation:
\begin{equation}\label{eq:reflection}
\Lr(u)+\Lr(1-u)=\Lr(1),\qquad 0<u<1,
\end{equation}
and Rogers' five-term relation:
\begin{equation}\label{eq:fiveterm}
\Lr(u)+\Lr(v)=\Lr(uv)
+\Lr\!\left(\frac{u(1-v)}{1-uv}\right)
+\Lr\!\left(\frac{v(1-u)}{1-uv}\right),
\qquad 0<u,v<1.
\end{equation}
Taking $u=v$ gives the duplication formula:
\begin{equation}\label{eq:dup}
\Lr(u^2)=2\Lr(u)-2\Lr\!\left(\frac{u}{1+u}\right).
\end{equation}

The problem belongs to the Nahm-sum program relating modular
$q$-hypergeometric series, algebraic equations, and dilogarithm identities.
Zagier's 2007 account gave a standard formulation of the problem, recorded
the rank-one classification of Terhoeven and Zagier, and listed eleven
rank-two candidate sets from a computer search~\cite{Zagier}; Vlasenko and
Zwegers subsequently
developed the asymptotic method in higher rank~\cite{VlasenkoZwegers}.
Kanade's search for modular companions was motivated in part by identities
and conjectures of Kur\c{s}ung\"oz and the Kanade--Russell
program~\cite{Kanade,Kursungoz}.  Mizuno later extended the framework to
symmetrizable matrices~\cite{Mizuno}.  One calculation in Kanade's paper led to the following dilogarithm problem, which was left open.

\begin{conjecture}[Kanade]\label{prop:kanadeconj}
Let
\begin{equation}\label{eq:Qdef}
Q_1=1-2\sin\frac{\pi}{18},\qquad
Q_2^3=4\sin^2\frac{\pi}{18}+4\sin\frac{\pi}{18},
\end{equation}
where $Q_2\in(0,1)$.  Then
\begin{equation}\label{eq:kanade}
\Lr(Q_1)+\frac13\Lr(Q_2^3)=\frac{4\pi^2}{27}.
\end{equation}
\end{conjecture}

The first purpose of this paper is to prove Conjecture~\ref{prop:kanadeconj}.
Kanade asked whether the identity could be related to the $\pi/9$
dilogarithm identities of Loxton and Lewin.  Kirillov gave an algebraic
derivation of those identities from the five-term relation~\cite[\S1.2.2]{Kirillov}.
The organization below follows Kirillov's idea.  The complete set of
Kirillov relations used here is stated explicitly, and the subsequent
elimination and cancellation are carried out in full, so the reader need
not reconstruct any step from the cited paper.

Next we use the rank-two exponent equations to derive a previously
unrecorded weight-$19$ analogue of~\eqref{eq:kanade}.  A second identity follows
from this relation and Kanade's identity.  Section~\ref{sec:qseries} briefly
discusses their relation to Mizuno's symmetrizable Nahm sums.

\section{Proof of Kanade's conjecture}

The proof is arranged in two stages.  First, the elementary five-term array
underlying the argument is written out and reduced.  Second, Loxton's known
$\pi/9$ identity is inserted and the remaining cancellation is performed
explicitly.

\begin{lemma}[Algebraic setup]\label{lem:algsetup}
Set:
\begin{equation}\label{eq:xyz}
x=\frac12\sec\frac{\pi}{9},\qquad
y=\frac{1}{1+x}=\frac12\sec\frac{2\pi}{9},\qquad
z=\frac{x}{1+x}=1-y,
\end{equation}
and introduce the auxiliary quantities:
\begin{equation}\label{eq:aux}
a=\frac{1}{x+2},\qquad
b=\frac{x^2}{1+x},\qquad
c=\frac{1}{x(x+2)},\qquad
h=\frac{x}{1+2x}.
\end{equation}
The quantities $a,c,h$ serve only to name intermediate arguments in the
five-term array below; they disappear when the array is reduced.
\end{lemma}

\begin{lemma}[Kirillov's five-term array]\label{lem:kirarray}
With the notation of Lemma~\ref{lem:algsetup}, the complete subarray needed
for Kanade's identity is:
\begin{align}
\Lr(y^2)&=2\Lr(y)-2\Lr(a),\label{eq:k12}\\
\Lr(a^2)&=2\Lr(a)-2\Lr(x^2),\label{eq:k13}\\
\Lr(z^2)&=2\Lr(z)-2\Lr(h),\label{eq:k14}\\
\Lr(y)+\Lr(a)&=\Lr(h)+\Lr(x)+\Lr(b),\label{eq:k15}\\
\Lr(x)+\Lr(x^2)&=\Lr(x^3)+\Lr(a^2)+\Lr(1)-\Lr(c^2),\label{eq:k16}\\
\Lr(c^2)&=2\Lr(c)-2\Lr(y^2),\label{eq:k17}\\
\Lr(c)&=\Lr(y)+\frac12\Lr(z^2).\label{eq:k18}
\end{align}
\end{lemma}

\begin{proof}
These seven identities are the relevant part of Kirillov's displayed array
in~\cite[\S1.2.2, pp.~78--79]{Kirillov}.  They are reproduced here so that
the elimination used below can be checked directly from the displayed
relations.
\end{proof}

The auxiliary values are most easily removed one at a time.  We begin with
those that occur in only two rows.  From~\eqref{eq:k12} and
\eqref{eq:k13}:
\[
\Lr(a)=\Lr(y)-\frac12\Lr(y^2),
\qquad
\Lr(a^2)=2\Lr(y)-\Lr(y^2)-2\Lr(x^2).
\]
Likewise, \eqref{eq:k18} and \eqref{eq:k17} give:
\[
\Lr(c)=\Lr(y)+\frac12\Lr(z^2),
\qquad
\Lr(c^2)=2\Lr(y)+\Lr(z^2)-2\Lr(y^2).
\]
Finally, \eqref{eq:k14} gives:
\[
\Lr(h)=\Lr(z)-\frac12\Lr(z^2).
\]
The auxiliary quantities can now simply be substituted away.  Substituting
the displayed expressions for $\Lr(a^2)$ and $\Lr(c^2)$ into
\eqref{eq:k16} gives:
\begin{equation}\label{eq:reduced1}
\Lr(x^3)-3\Lr(x^2)-\Lr(x)
=\Lr(z^2)-\Lr(y^2)-\Lr(1).
\end{equation}
Substituting the displayed expressions for $\Lr(a)$ and $\Lr(h)$ into
\eqref{eq:k15} gives:
\begin{equation}\label{eq:reduced2}
2\Lr(y)-\frac12\Lr(y^2)
=\Lr(z)-\frac12\Lr(z^2)+\Lr(x)+\Lr(b).
\end{equation}
Thus $a,a^2,c,c^2,h$ have disappeared by direct substitution, leaving only
\eqref{eq:reduced1}--\eqref{eq:reduced2}.

The remaining external input is the following proved $\pi/9$ identity of
Loxton~\cite{Loxton,Lewin}:
\begin{equation}\label{eq:loxton}
\Lr(x^3)-3\Lr(x^2)-3\Lr(x)=-\frac73\Lr(1).
\end{equation}
\begin{theorem}\label{thm:kanade}
Kanade's identity~\eqref{eq:kanade} holds.
\end{theorem}

\begin{proof}
The $x^3$ term can now be removed.  Subtracting
\eqref{eq:loxton} from~\eqref{eq:reduced1} and rearranging gives:
\begin{equation}\label{eq:diffsq}
\Lr(y^2)-\Lr(z^2)=\frac43\Lr(1)-2\Lr(x).
\end{equation}
Rearranging~\eqref{eq:reduced2} and then using~\eqref{eq:diffsq} gives:
\begin{align}
2\Lr(y)-\Lr(b)
&=\Lr(x)+\Lr(z)+\frac12\bigl(\Lr(y^2)-\Lr(z^2)\bigr)\notag\\
&=\Lr(z)+\frac23\Lr(1).\label{eq:intermediate}
\end{align}
Since $y+z=1$, Euler's relation gives $\Lr(z)=\Lr(1)-\Lr(y)$, and hence:
\begin{equation}\label{eq:keykanade}
3\Lr(y)-\Lr(b)=\frac53\Lr(1).
\end{equation}
This is the entire dilogarithmic cancellation.  Applying Euler's relation
once more:
\begin{equation}\label{eq:kanadefinish}
\Lr(y)+\frac13\Lr(1-b)
 =\frac13\bigl(3\Lr(y)-\Lr(b)+\Lr(1)\bigr)
 =\frac89\Lr(1)=\frac{4\pi^2}{27}.
\end{equation}

It remains only to identify the arguments with Kanade's notation.  From
the definitions:
\[
Q_1=\frac{1}{1+x}=y,
\qquad
Q_2^3=1-\frac{x^2}{1+x}=1-b.
\]
Thus \eqref{eq:kanadefinish} is exactly \eqref{eq:kanade}.
\end{proof}

\noindent We henceforth write:
\begin{equation}\label{eq:jbeta}
j:=y=\frac12\sec\frac{2\pi}{9},\qquad
\beta:=b=\frac{(1-j)^2}{j}.
\end{equation}
Then~\eqref{eq:keykanade} is equivalently:
\begin{equation}\label{eq:jbetaid}
3\Lr(j)-\Lr(\beta)=\frac{5\pi^2}{18}.
\end{equation}
The number $j$ satisfies:
\begin{equation}\label{eq:jcubic}
j^3-3j^2+1=0.
\end{equation}

\section{A rank-two exponent-matrix reduction}\label{sec:matrix}

We use the following two-variable form of Nahm's equations, written with a common denominator $u$ so that matrices with rational entries are covered. Zagier treats the symmetric positive-definite setting in~\cite{Zagier}; Mizuno extends it to symmetrizable matrices~\cite{Mizuno}.
Let $F$ be a field containing $z_0$ and $z_1$, and work in
$\bigwedge^2(F^\times\otimes_{\mathbb Z}\mathbb Q)$ so that rational
exponents may be handled linearly.  The wedge product is bilinear and
alternating, so $v\wedge v=0$ and $v\wedge w=-w\wedge v$.  Write:
\begin{equation}\label{eq:exponentmatrix}
B=\frac1u\begin{pmatrix}a&M\\ m&d\end{pmatrix}
\end{equation}
and consider:
\begin{equation}\label{eq:exponentsys}
1-z_0=z_0^{a/u}z_1^{M/u},\qquad
1-z_1=z_0^{m/u}z_1^{d/u},\qquad 0<z_0,z_1<1.
\end{equation}

\begin{lemma}[Wedge cancellation]\label{lem:wedge}
Every solution of~\eqref{eq:exponentsys} satisfies:
\begin{equation}\label{eq:wedgecancel}
m\,z_0\wedge(1-z_0)+M\,z_1\wedge(1-z_1)=0.
\end{equation}
Indeed,
\[
z_0\wedge(1-z_0)=\frac{M}{u}\,z_0\wedge z_1,
\qquad
z_1\wedge(1-z_1)=-\frac{m}{u}\,z_0\wedge z_1,
\]
and multiplying by $m$ and $M$, respectively, gives~\eqref{eq:wedgecancel}.
Thus the corresponding two-term Rogers dilogarithm expression has the form:
\begin{equation}\label{eq:twotermform}
m\Lr(z_0)+M\Lr(z_1).
\end{equation}
\end{lemma}

\noindent Take the diagonal matrix:
\begin{equation}\label{eq:symmetrizergeneral}
D=\operatorname{diag}(M,m).
\end{equation}
Then $BD$ is symmetric.  When $BD$ is positive definite, the pair $(B,D)$
is of the type used in Mizuno's symmetrizable Nahm sums.  The first two
examples below have $M=1$; the third has $M=3$.

\begin{lemma}[One-variable reduction]\label{lem:exponentreduce}
Assume $M=1$ in \eqref{eq:exponentmatrix}.  Then:
\begin{equation}\label{eq:z1eliminate}
z_1=\frac{(1-z_0)^u}{z_0^a},
\end{equation}
and every solution of \eqref{eq:exponentsys} satisfies:
\begin{equation}\label{eq:scalarpoly}
z_0^a-(1-z_0)^u-(1-z_0)^d z_0^p=0,
\qquad
p=\frac{au+m-da}{u}.
\end{equation}
\end{lemma}

\begin{proof}
The first equation of \eqref{eq:exponentsys} gives
\eqref{eq:z1eliminate}.  Substituting it into the second equation and
multiplying by $z_0^a$ gives \eqref{eq:scalarpoly}.
\end{proof}

\subsection{Kanade's cubic}

\noindent Kanade's identity corresponds to:
\begin{equation}\label{eq:BK}
B_K=\begin{pmatrix}2&1\\3&2\end{pmatrix},\qquad
D_K=\operatorname{diag}(1,3).
\end{equation}
The product
\[
B_KD_K=\begin{pmatrix}2&3\\3&6\end{pmatrix}
\]
is symmetric positive definite.  Since $(m,M)=(3,1)$, Lemma~\ref{lem:wedge}
selects the two-term combination $3\Lr(z_0)+\Lr(z_1)$, which is three times
the normalization in Kanade's identity.  The corresponding equations are:
\begin{equation}\label{eq:Kanadeexponentsys}
1-z_0=z_0^2z_1,\qquad
1-z_1=z_0^3z_1^2.
\end{equation}
Lemma~\ref{lem:exponentreduce}, with $(u,a,m,d)=(1,2,3,2)$, gives:
\[
z_0^2-(1-z_0)-z_0(1-z_0)^2
=-(z_0^3-3z_0^2+1).
\]
Thus $z_0=j=Q_1$.  Moreover, using~\eqref{eq:jcubic}:
\[
z_1=\frac{1-j}{j^2}=1-\frac{(1-j)^2}{j}=Q_2^3.
\]
Hence the two-variable system recovers exactly the two arguments in
Kanade's identity.

\section{A weight-\texorpdfstring{$19$}{19} analogue and a second identity}
\label{sec:further}

\noindent Let $j$ be as in~\eqref{eq:jbeta}, and define:
\begin{equation}\label{eq:alphabeta}
\alpha=j^2(1-j)^3=3-7j^2,
\qquad
\beta=\frac{(1-j)^2}{j}=-2+4j-j^2.
\end{equation}
They lie in $(0,1)$ and have minimal polynomials:
\begin{equation}\label{eq:alphabetaminpoly}
\alpha^3+54\alpha^2-57\alpha+1=0,
\qquad
\beta^3+3\beta^2-6\beta+1=0.
\end{equation}

The first identity below is exact.  Its proof is a
finite Kirillov-style calculation and is placed in \ref{app:m19}
so that the main discussion is not interrupted by the larger cancellation
array.

\begin{theorem}[The weight-$19$ identity]\label{thm:m19}
Let $j$ be the root in $(0,1)$ of $j^3-3j^2+1=0$, and let
$\alpha=j^2(1-j)^3$.  Then
\begin{equation}\label{eq:m19identity}
19\Lr(j)+\Lr(\alpha)=2\pi^2.
\end{equation}
\end{theorem}

\begin{proof}
See \ref{app:m19}.  There the identity is obtained from eleven
distinct specializations of Rogers' five-term relation, followed by nine
Euler reflections and one explicit integer linear combination.
\end{proof}

\noindent Euler reflection gives an equivalent form with an exponent matrix having
$M=1$.  Set $r=1-j$ and $s=1-\alpha$.  Then Theorem~\ref{thm:m19}
is equivalent to:
\begin{equation}\label{eq:m19reflected}
19\Lr(r)+\Lr(s)=\frac{4\pi^2}{3},
\end{equation}
with exponent matrix:
\begin{equation}\label{eq:B1}
B_1=\frac15\begin{pmatrix}2&1\\19&2\end{pmatrix}.
\end{equation}
Since $(m,M)=(19,1)$, Lemma~\ref{lem:wedge} gives the coefficient pattern
$19\Lr(r)+\Lr(s)$ in~\eqref{eq:m19reflected}.  For
$(u,a,m,d)=(5,2,19,2)$, Lemma~\ref{lem:exponentreduce} gives:
\[
s=\frac{(1-r)^5}{r^2}
\]
and the one-variable equation:
\begin{align}
&r^2-(1-r)^5-(1-r)^2r^5 \notag\\
&\qquad=-(r^2-r+1)^2(r^3-3r+1)=0.\label{eq:m19cleanfactor}
\end{align}
Since $r=1-j$ satisfies $r^3-3r+1=0$, the relevant factor is the cubic.
A direct substitution using $j^3-3j^2+1=0$ also gives:
$s=j^5/(1-j)^2=1-\alpha$.  Thus
\eqref{eq:B1} encodes the reflected weight-$19$ identity with no auxiliary
high-degree polynomial.

The second identity follows from Theorem~\ref{thm:m19} and
the already proved Kanade relation~\eqref{eq:jbetaid}.

\begin{corollary}\label{cor:m19second}
With $\alpha$ and $\beta$ as in~\eqref{eq:alphabeta},
\begin{equation}\label{eq:m19second}
19\Lr(\beta)+3\Lr(\alpha)=\frac{13\pi^2}{18}.
\end{equation}
\end{corollary}

\begin{proof}
Multiply~\eqref{eq:m19identity} by $3$ and multiply
\eqref{eq:jbetaid} by $19$.  Subtracting the latter equality from the
former gives:
\[
19\Lr(\beta)+3\Lr(\alpha)
=6\pi^2-\frac{95\pi^2}{18}
=\frac{13\pi^2}{18}.
\]
\end{proof}

\noindent Reflect both arguments in the second identity and set
$r=1-\beta$, $s=1-\alpha$.  Equation~\eqref{eq:m19second} becomes:
\begin{equation}\label{eq:m19secondreflected}
19\Lr(r)+3\Lr(s)=\frac{53\pi^2}{18}.
\end{equation}
Using $j^3-3j^2+1=0$ in the definitions of $\alpha$ and $\beta$ gives:
\begin{equation}\label{eq:B2sys}
1-r=r^8s^3,\qquad
1-s=r^{19}s^8,
\end{equation}
with exponent matrix:
\begin{equation}\label{eq:B2}
B_2=\begin{pmatrix}8&3\\19&8\end{pmatrix}.
\end{equation}
Here $(m,M)=(19,3)$, and Lemma~\ref{lem:wedge} gives exactly the coefficient
pattern $19\Lr(r)+3\Lr(s)$ in~\eqref{eq:m19secondreflected}.  The first
coordinate is exactly Kanade's second argument, $r=1-\beta=Q_2^3$; it has
minimal polynomial:
\[
r^3-6r^2+3r+1=0.
\]
Equations \eqref{eq:m19cleanfactor} and \eqref{eq:B2sys} are the only
matrix calculations needed for the two additional identities.

\section{Numerical modularity check}\label{sec:qseries}

\noindent Kanade's identity arose from a Nahm-type $q$-series.  We
therefore check whether the two weight-$19$ evaluations produce a modular
candidate in Mizuno's symmetrizable framework~\cite{Mizuno}, using his
public implementation~\cite{MizunoCode}.

For the first reflected realization~\eqref{eq:B1}, the natural diagonal
matrix $D_1=\operatorname{diag}(1,19)$ gives:
\[
B_1D_1=\frac15\begin{pmatrix}2&19\\19&38\end{pmatrix},
\qquad
\det(B_1D_1)=-\frac{57}{5}<0 .
\]
Since $B_1D_1$ is symmetric and has negative determinant, it has one positive
and one negative eigenvalue.  Hence it is indefinite and does not satisfy
the positive-definiteness condition in Mizuno's search.  This does not by
itself make the associated series divergent: every entry of $B_1D_1$ is
positive, so the quadratic form is nonnegative on the summation cone.  The
finite-difference test can still be applied numerically, and we report the
outcome below, with the caveat that the asymptotic analysis underlying that
test is proved only in the positive-definite case.

The Neumann--Zagier system attached to $(B_1,D_1)$ is
\eqref{eq:exponentsys}, with
\[
(u,a,m,d)=(5,2,19,2).
\]
Its solution is $(r,s)=(1-j,1-\alpha)$ as in~\eqref{eq:m19reflected}.
The effective central charge is:
\[
\frac{6}{\pi^2}\bigl(19\Lr(j)+\Lr(\alpha)\bigr)=12 ,
\]
by Theorem~\ref{thm:m19}.  Mizuno's implementation locates this solution by a
Newton iteration from a fixed starting point; that iteration does not converge
on the present input, so the value above is obtained from the solution
$(r,s)=(1-j,1-\alpha)$ directly.

The second reflected realization uses:
\begin{equation}\label{eq:D2}
D_2=\operatorname{diag}(3,19),
\end{equation}
for which:
\begin{equation}\label{eq:B2sym}
B_2D_2=\begin{pmatrix}24&57\\57&152\end{pmatrix},
\qquad
\det(B_2D_2)=399>0 .
\end{equation}
Hence $B_2D_2$ is symmetric positive definite and passes the first test in
Mizuno's implementation.  The Neumann--Zagier system used by the
central-charge routine is:
\[
w_0=(1-w_0)^8(1-w_1)^3,
\qquad
w_1=(1-w_0)^{19}(1-w_1)^8 ,
\]
which with $w_0=\beta$ and $w_1=\alpha$ is precisely the reflected
system~\eqref{eq:B2sys}.  The effective central charge is:
\[
\frac{6}{\pi^2}\bigl(19\Lr(\beta)+3\Lr(\alpha)\bigr)=\frac{13}{3},
\]
by Corollary~\ref{cor:m19second}.  The pair therefore also passes
Mizuno's rational central-charge test.  We note in passing that the two
central charges, $12$ and $13/3$, are exactly the two identities of
Section~\ref{sec:further} in normalized form.

\medskip
\noindent\textbf{The searches.}
The driver published with~\cite{MizunoCode} reports only those candidates that
meet its cutoff, and prints nothing when a search finds none.  We therefore
called its rank-two routines directly, at the constants fixed in that source:
for a single linear term, $N=21,\ldots,24$ with the summation truncated at
$150$ in each index; for the paired test below, $N=18,19,20$ with truncation
$75$.  The linear terms searched are those of the published driver,
\[
b=(b_0/q,b_1/q),\qquad -4\le b_0,b_1\le10,\quad 1\le q\le4 .
\]
No choice of $b$ met the program's modularity cutoff of $10^{-10}$ in both
finite-difference residuals, for either realization.  The smallest maximum
residual attained was $5.3\times10^{-7}$ for $(B_2,D_2)$ and
$1.8\times10^{-7}$ for $(B_1,D_1)$, in both cases more than three orders of
magnitude above the cutoff.

Mizuno's implementation also provides a paired search, which tests the
combination $f_{A,b_1}+3f_{A,b_2}$ over ordered pairs $(b_1,b_2)$ drawn
from the same range, since a modular companion may exist for a combination
of two Nahm sums when it exists for neither separately.  Over the
resulting $8\times10^5$ pairs, no combination met the cutoff for either
realization; the closest were $1.6\times10^{-10}$ for $(B_2,D_2)$ and
$5.3\times10^{-10}$ for $(B_1,D_1)$.  Repeating the sweep with the
combination coefficient $3$ replaced by each of
$1,2,4,5,6,9,19,\tfrac13,\tfrac{3}{19},\tfrac{19}{3}$ produced occasional
values below the cutoff, but the paired test evaluates a single second
difference from three values of $N$, and at that resolution such values
occur by chance at this sample size.  Every one of them was retested at
the constants of the single-$b$ test, with five values of $N$ and three
second differences; all failed by four to six orders of magnitude.

\medskip
\noindent This is numerical evidence only and does not exclude a modular
companion arising from a different $q$-series construction.

\medskip
\noindent\textbf{Verification.}
Two Python programs accompany this paper. The first,
\path{verify_dilogarithm_identities.py}, checks the dilogarithm identities,
the algebraic factorizations, and the certificate in
\ref{app:m19}, with the latter verified exactly in
$\mathbb{Q}(j)=\mathbb{Q}[t]/(t^3-3t^2+1)$. The second,
\path{verify_mizuno_modularity.py}, reproduces the numerical calculations
reported in Section~\ref{sec:qseries}, including the Neumann--Zagier solution,
effective central charge, and modularity searches. Its rank-two routines
follow those in Mizuno's public implementation~\cite{MizunoCode}.

\medskip
\noindent\textbf{AI disclosure.}
ChatGPT, using GPT-5.6 Sol, was used to assist in creating the verification
scripts for the dilogarithm certificates and modularity checks, and in
carrying out the numerical modularity tests.

\appendix
\section{Explicit five-term proof of the weight-$19$ identity}
\label{app:m19}

We prove Theorem~\ref{thm:m19} by a finite five-term certificate.  Set:
\begin{equation}\label{eq:xylambda}
\eta=1-j,\qquad
\lambda=2-j^{-1}.
\end{equation}
Since $j^3-3j^2+1=0$ and $j\ne0$:
\begin{equation}\label{eq:xinverse}
j^{-1}=3j-j^2.
\end{equation}
Also:
\begin{equation}\label{eq:alphaXY}
\alpha=j^2\eta^3.
\end{equation}
All algebraic simplifications below use $\eta=1-j$, $\lambda=2-j^{-1}$,
and~\eqref{eq:xinverse}.  The elementary bounds
$0.1736<\sin(\pi/18)<0.17365$ give $0.6527<j<0.6528$.
Interval evaluation of the distinct inputs occurring on the left of
$F_1,\ldots,F_{11}$ shows that every such input lies between $0.051$ and $0.968$.
Hence the hypotheses $0<u,v<1$ of the five-term relation are satisfied,
and the right-hand arguments are then automatically in $(0,1)$.  Every entry
of the table in Lemma~\ref{lem:m19pairs} occurs among these inputs or among the
right-hand arguments of $F_1,\ldots,F_{11}$, so each reflection row below also
has its argument in $(0,1)$.

\begin{lemma}[The $m=19$ five-term array]\label{lem:m19array}
With the notation above, the following eleven identities are instances of
Rogers' five-term relation~\eqref{eq:fiveterm}:
\begin{align}
\Lr(\lambda/j)+\Lr(j)
 &=2\Lr(\lambda)+\Lr(\eta),\tag{$F_1$}\label{eq:F1}\\
\Lr(\eta/j)+\Lr(j\eta)
 &=\Lr(\lambda)+2\Lr(\eta^2),\tag{$F_2$}\label{eq:F2}\\
\Lr(\eta/\lambda)+\Lr(\eta)
 &=\Lr(\eta^2/\lambda)+\Lr(j)+\Lr(\eta^2),\tag{$F_3$}\label{eq:F3}\\
\Lr(\eta/j^2)+\Lr(j^3/\eta)
 &=\Lr(j)+\Lr(\lambda)+\Lr(j^2),\tag{$F_4$}\label{eq:F4}\\
\Lr(\lambda/j)+\Lr(j^2/\lambda)
 &=\Lr(j)+\Lr(\eta^2/j)+\Lr(\eta/\lambda),\tag{$F_5$}\label{eq:F5}\\
\Lr\!\bigl(\eta^2/(j\lambda)\bigr)+\Lr(\lambda)
 &=\Lr(\eta^2/j)+\Lr(\eta^2/\lambda)+\Lr(\eta),\tag{$F_6$}\label{eq:F6}\\
2\Lr(j)
 &=\Lr(j^2)+2\Lr\!\bigl(\eta^2/(j\lambda)\bigr),\tag{$F_7$}\label{eq:F7}\\
\Lr(\lambda/j)+\Lr(\lambda)
 &=\Lr(\lambda^2/j)+\Lr(j^2\lambda/\eta)+\Lr(j^2\lambda),\tag{$F_8$}\label{eq:F8}\\
2\Lr(j\eta)
 &=\Lr(j^2\eta^2)+2\Lr(\eta^2/j),\tag{$F_9$}\label{eq:F9}\\
\Lr\!\bigl(\lambda^2/(j\eta)\bigr)+\Lr(\eta/\lambda)
 &=\Lr(\lambda/j)+\Lr(j\lambda/\eta)+\Lr(\eta^3/\lambda),\tag{$F_{10}$}\label{eq:F10}\\
\Lr(\eta)+\Lr(j^2\eta^2)
 &=\Lr(\alpha)+\Lr(\lambda^2/j)+\Lr(\eta^4/j^2).\tag{$F_{11}$}\label{eq:F11}
\end{align}
\end{lemma}

\begin{proof}
For each row, substitute the two arguments on the left into
\eqref{eq:fiveterm}.  To identify the three arguments on the right with
those displayed in $F_1,\ldots,F_{11}$, substitute $\eta=1-j$ and
$\lambda=2-j^{-1}$ and replace every occurrence of $j^{-1}$ by
$3j-j^2$ using~\eqref{eq:xinverse}.  For example, $F_1$ is the
specialization $(u,v)=(\lambda/j,j)$: its three five-term arguments are
\[
uv=\lambda,
\qquad
\frac{u(1-v)}{1-uv}=\lambda,
\qquad
\frac{v(1-u)}{1-uv}=\eta.
\]
The other ten rows are checked identically.  The repeated-argument rows
$F_7$ and $F_9$ are the $u=v$ specialization of the five-term relation,
i.e. the duplication formula~\eqref{eq:dup} written in the same format.
\end{proof}

The Euler reflections needed after the eleven-row cancellation are recorded
in the following table:

\begin{lemma}[Complementary pair table]\label{lem:m19pairs}
With $j,\eta,\lambda$ as above, the following nine pairs are complementary:

\begin{center}
\begin{tabular}{@{}cccc@{}}
\toprule
Row & $u$ & $1-u$ & Numerical values $(u,1-u)$\\
\midrule
$R_1$ & $\eta/j^2$          & $\eta^2/j$
      & $(0.815207,\;0.184793)$\\
$R_2$ & $\eta^4/j^2$        & $\lambda^2/(j\eta)$
      & $(0.034148,\;0.965852)$\\
$R_3$ & $\eta/j$            & $\lambda$
      & $(0.532089,\;0.467911)$\\
$R_4$ & $\eta/\lambda$      & $\eta^2/\lambda$
      & $(0.742227,\;0.257773)$\\
$R_5$ & $\eta$              & $j$
      & $(0.347296,\;0.652704)$\\
$R_6$ & $\eta^2$            & $j\lambda/\eta$
      & $(0.120615,\;0.879385)$\\
$R_7$ & $\eta^3/\lambda$    & $j^2/\lambda$
      & $(0.089524,\;0.910476)$\\
$R_8$ & $j^2\lambda/\eta$   & $j^2$
      & $(0.573978,\;0.426022)$\\
$R_9$ & $j^2\lambda$        & $j^3/\eta$
      & $(0.199340,\;0.800660)$\\
\bottomrule
\end{tabular}
\end{center}

Thus row $R_i$ gives the reflection identity:
\begin{equation}\label{eq:Ri}
\mathcal R_i:=\Lr(u_i)+\Lr(1-u_i)-\Lr(1)=0,
\qquad 1\le i\le 9.
\end{equation}
\end{lemma}

\begin{proof}
Let
\[
f(j)=j^3-3j^2+1.
\]
Substitute $\eta=1-j$ and $\lambda=2-j^{-1}$ into each row, and let
$\Delta_i$ denote the sum of the two displayed entries in row $R_i$ minus
$1$.  Direct simplification gives:
\[
\begin{aligned}
\Delta_1&=\frac{f(j)}{j^2},\\
\Delta_2&=\frac{f(j)(j^3-2j^2+3j-1)}{j^3(j-1)},\\
\Delta_3&=0,\\
\Delta_4&=\frac{f(j)}{2j-1},\\
\Delta_5&=0,\\
\Delta_6&=\frac{f(j)}{j-1},\\
\Delta_7&=-\frac{(j-1)f(j)}{2j-1},\\
\Delta_8&=\frac{f(j)}{j-1},\\
\Delta_9&=\frac{f(j)}{j-1}.
\end{aligned}
\]
Since $f(j)=0$, all nine quantities $\Delta_i$ vanish.  Hence every pair
in the table is complementary.
\end{proof}

\begin{proof}[Proof of Theorem~\ref{thm:m19}]
For $1\le i\le 11$, let $\mathcal F_i$ denote the left side of $F_i$
minus its right side, so that $\mathcal F_i=0$.  The required linear
combination of the eleven five-term rows is:
\begin{align}
\Sigma_F
&:=\mathcal F_1+2\mathcal F_2-3\mathcal F_3-\mathcal F_4-\mathcal F_5
   +4\mathcal F_6+2\mathcal F_7+\mathcal F_8-\mathcal F_9
   +\mathcal F_{10}-\mathcal F_{11} \notag\\
&=10\Lr(j)+2\Lr(\lambda)-9\Lr(\eta)+2\Lr(\eta/j)-\Lr(\eta^2)
   -\Lr(\eta/\lambda)-\Lr(\eta^2/\lambda) \notag\\
&\quad-\Lr(\eta/j^2)-\Lr(j^3/\eta)-\Lr(j^2)-\Lr(j^2/\lambda)
   -\Lr(\eta^2/j)-\Lr(j^2\lambda/\eta)-\Lr(j^2\lambda) \notag\\
&\quad+\Lr(\lambda^2/(j\eta))-\Lr(j\lambda/\eta)-\Lr(\eta^3/\lambda)
   +\Lr(\alpha)+\Lr(\eta^4/j^2)=0. \label{eq:m19SigmaF}
\end{align}

Using the numbered reflection rows $R_1,\ldots,R_9$ from
Lemma~\ref{lem:m19pairs}, form the second vanishing combination:
\begin{align}
\Sigma_R
&:=\mathcal R_1-\mathcal R_2-2\mathcal R_3+\mathcal R_4
   +9\mathcal R_5+\mathcal R_6+\mathcal R_7+\mathcal R_8+\mathcal R_9 \notag\\
&=\Lr(\eta/j^2)+\Lr(\eta^2/j)-\Lr(\eta^4/j^2)-\Lr(\lambda^2/(j\eta))
   -2\Lr(\eta/j)-2\Lr(\lambda) \notag\\
&\quad+\Lr(\eta/\lambda)+\Lr(\eta^2/\lambda)+9\Lr(\eta)+9\Lr(j)
   +\Lr(\eta^2)+\Lr(j\lambda/\eta) \notag\\
&\quad+\Lr(\eta^3/\lambda)+\Lr(j^2/\lambda)+\Lr(j^2\lambda/\eta)
   +\Lr(j^2)+\Lr(j^2\lambda)+\Lr(j^3/\eta)-12\Lr(1)=0.
   \label{eq:m19SigmaR}
\end{align}
Adding~\eqref{eq:m19SigmaF} and~\eqref{eq:m19SigmaR} cancels every
auxiliary dilogarithm value and gives:
\begin{equation}\label{eq:m19finalcancel}
\Sigma_F+\Sigma_R
=19\Lr(j)+\Lr(\alpha)-12\Lr(1)=0.
\end{equation}
Therefore:
\[
19\Lr(j)+\Lr(\alpha)=12\Lr(1)=2\pi^2.
\]
\end{proof}


\end{document}